\documentclass[reqno,11pt]{amsart}
\usepackage{amsmath,amsfonts,amscd,amssymb,latexsym}

\usepackage{hyperref}
\usepackage{soul}

\DeclareFontFamily{U}{BOONDOX-calo}{\skewchar\font=45 }
\DeclareFontShape{U}{BOONDOX-calo}{m}{n}{
  <-> s*[1.05] BOONDOX-r-calo}{}
\DeclareFontShape{U}{BOONDOX-calo}{b}{n}{
  <-> s*[1.05] BOONDOX-b-calo}{}
\DeclareMathAlphabet{\mathcalboondox}{U}{BOONDOX-calo}{m}{n}
\SetMathAlphabet{\mathcalboondox}{bold}{U}{BOONDOX-calo}{b}{n}
\DeclareMathAlphabet{\mathbcalboondox}{U}{BOONDOX-calo}{b}{n}

\usepackage{mathrsfs}  

\usepackage{accents}

\usepackage{comment}
\usepackage{soul}
\usepackage{tikz}
\usepackage{marginnote}
\usepackage{anysize}
\usepackage{bbm}
\usepackage{cancel}
\usepackage[normalem]{ulem}

\numberwithin{equation}{section}

\theoremstyle{plain}
\newtheorem{theorem}[subsection]{Theorem}
\newtheorem{lemma}[subsection]{Lemma}
\newtheorem{corollary}[subsection]{Corollary}
\newtheorem{proposition}[subsection]{Proposition}

\theoremstyle{definition}
\newtheorem{definition}[subsection]{Definition}

\theoremstyle{remark}
\newtheorem{remark}[subsection]{Remark}

\newcommand{\ind}{\operatorname{ind}}
\newcommand{\id}{\operatorname{Id}}
\newcommand{\dom}{\operatorname{dom}}
\newcommand{\loc}{\operatorname{loc}}

\newcommand{\RR}{\mathbb{R}}\newcommand{\AAA}{\mathbb{A}}

\newcommand{\ZZ}{\mathbb{Z}}

\newcommand{\DD}{\mathcal{D}}\newcommand{\cDD}{\check{\mathcal{D}}}
\newcommand{\cbeta}{\check{\beta}}
\newcommand{\cgamma}{\check{\gamma}}

\newcommand{\bfu}{\mathbf{u}}
\newcommand{\bfv}{\mathbf{v}}
\newcommand{\bfw}{\mathbf{w}}
\newcommand{\bff}{\mathbf{f}}
\newcommand{\bfg}{\mathbf{g}}
\newcommand{\bfh}{\mathbf{h}}
\newcommand{\upper}{\uppercase\expandafter}

\newcommand{\n}{\nabla}
\newcommand{\p}{\partial}
\newcommand{\pM}{{\p M}}

\newcommand{\supp}{\operatorname{supp}}

\newcommand{\End}{\operatorname{End}}

\newcommand{\spf}{\operatorname{sf}}

\renewcommand{\AA}{\mathbb{A}}
\newcommand{\IM}{\operatorname{Im}}

\newcommand{\oeta}{\bar{\eta}}
\newcommand{\res}{\operatorname{res}}

\newcommand{\Q}{\mathcal{Q}}

\newcommand{\E}{\mathcal{E}}

\newcommand{\ch}{\operatorname{ch}}

\newcommand{\TAS}{\operatorname{T\hspace{.05em}\alpha}_{AS}}

\newcommand{\dist}{\operatorname{dist}}

\newcommand{\<}{\langle}
\renewcommand{\>}{\rangle}
\newcommand{\RE}{\operatorname{Re}}

\newcommand{\comp}{\mathrm{comp}}
\renewcommand{\AA}{\mathcal{A}}

\newcommand{\HAs}[1]{H_{\AA_{#1}}^s(\Sigma_{#1},E_{#1})}
\newcommand{\HAsp}[1]{H_{\AA_{#1}}^s(\Sigma_{#1},E^+|_{\Sigma_{#1}})}
\newcommand{\LS}[1]{L^2(\Sigma_{#1},E_{#1})}
\newcommand{\LSI}[2]{L^2_{#2}(\Sigma_{#1},E_{#1})}

\begin{document}

\author[Maxim Braverman]{Maxim Braverman${}^\dag$}
\address{Department of Mathematics,
Northeastern University,
Boston, MA 02115,
USA}

\email{m.braverman@northeastern.edu}
\urladdr{www.math.neu.edu/~braverman/}

\thanks{${}^\dag$†Partially supported by the Simons Foundation collaboration grant \#G00005104.}

\title[Index on Lorentzian manifolds]{An index of strongly Callias operators on Lorentzian manifolds with non-compact boundary}

\begin{abstract}
We consider hyperbolic Dirac-type operator with growing potential on a spatially non-compact globally hyperbolic manifold. We show that the Atiyah-Patodi-Singer boundary value problem for such operator is Fredholm and obtain a formula for this index in terms of the local integrals and the relative eta-invariant introduced by Braverman and Shi. This extends recent results of B\"ar and Strohmaier, who studied the index of a hyperbolic Dirac operator on a spatially compact globally hyperbolic manifold. 
\end{abstract}

\maketitle


\section{Introduction}\label{S:introduction}

Recently  B\"ar and Strohmaier, \cite{BarStrohmaier15}, discovered that a Dirac operator $D$ with Atiyah-Patodi-Singer (APS) boundary conditions on a spatially compact globally hyperbolic manifold is Fredholm. This is quite surprising, since the operator is not elliptic. B\"ar and Strohmaier showed that the index of the APS boundary value  problem for such operator is equal to the index of the APS boundary problem for an elliptic operator, obtained from $D$ by a ``Wick rotation". Thus this index can be computed by the usual APS index theorem \cite{APS123}. 

Besides significant contribution to the index theory, the result of B\"ar and Strohmaier provides the first mathematically rigorous description of chiral anomaly in quantum field theory, \cite{BarStrohmaier16}, but only in spatially compact case. It is desirable to extend the results of  \cite{BarStrohmaier15} to spatially non-compact manifolds, thus, in particular, providing a mathematically rigorous  description of anomalies in  more  realistic physical situations. The current paper is an attempt to do so by studying so called strongly Callias-type operators on spatially non-compact globally hyperbolic manifolds. 

A systematic treatment of the index of boundary value problems for strongly Callias-type operators on non-compact manifolds with non-compact boundary was given in \cite{BrShi17,BrShi17b,Shi18,BrShi18}. In particular, in \cite{BrShi17,BrShi17b} an APS-type index formula is obtained for  strongly Callias-type operators on so called  {\em essentially cylindrical manifolds} -- manifolds, which outside of a compact set look like a product $[0,1]\times Y$. 

In this paper we consider a (non-compact) manifold $M= [0,1]\times \Sigma$ and endow it with a Lorentzian metric $\<\cdot,\cdot\>$,  which is product outside of a compact set. Let $E=E^+\oplus E^-$ be a graded Dirac bundle over $M$ and let $D$ be the corresponding Dirac operator. 
A {\em strongly Callias-type operator} is the operator $\DD:=D+F$, where $F$ is a self-adjoint bundle map (called the {\em Callias potential}) which anticommutes with the Cliford multiplication and satisfies certain growth  conditions at infinity. 
On manifolds without boundary these conditions guarantee that the spectrum of $\DD$ is descrete. 
This implies, in particular,  that the spectrum of the restriction $\AA_t$ of $\DD$ to each space-like hypersurface $\Sigma_t=\{t\}\times \Sigma$ is discrete. We also assume that {\em there exists a compact subset $K\subset \Sigma$, such that the restriction of all the structures  to the complement of\, $[0,1]\times K$ are product}. The first main result of the paper is that the APS boundary value problem of $\DD$ is Fredholm. This extends a result of \cite{BarStrohmaier15} to our non-compact situation. 
Next, we show that the index of this boundary value problem is equal to the APS index of the elliptic strongly Callias-type operator $\cDD$, obtained from $\DD$ by Wick rotation. This allows us to compute this index by an APS-type index formula (with eta-invariant replaced with the relative eta-invariant introduced in \cite{BrShi17,BrShi17b}).

Our proof of Fredholmness of the APS boundary value problem for $\DD$ is quite different from that in  \cite{BarStrohmaier15}, because we need to study the behavior of the solutions of $\DD\bfu=0$ at infinity. Let us discuss the main steps of the proof.

\subsection{The wave evolution operator}\label{SS:Ievolution}
In Section~\ref{S:evolution} we carefully study the behavior of the solutions of the wave equation $\DD\bfu=\bff$ at infinity. This allows us do define  spaces of sections of $E$, in which the inhomogeneous  Cauchy problem for $\DD$ is well-posed.  As a consequence, we define the unitary {\em wave evolution operator} $Q:L^2(\Sigma_0,E^+_0)\to L^2(\Sigma_1,E^+_1)$, where $E_t^+$ denotes the restriction of $E^+$ to $\Sigma_t:=\{t\}\times\Sigma\subset M$. 

Let $\AA_t$ denote the restriction of $\DD$ to $\Sigma_t$.  Following \cite{BarStrohmaier15} we decompose the space of $L^2$-sections over $\Sigma_t$ into the direct sum of the spectral subspaces of $\AA_t$ corresponding to positive and negative part of the spectrum. We write $Q$ as a matrix 
\begin{equation}\label{E:IQmatrix}
	Q\ = \ \begin{pmatrix}
	Q_{++}&Q_{+-}\\Q_{-+}&Q_{--}
	\end{pmatrix}
\end{equation}
with respect to this decomposition. A key resut here is that the operators $Q_{++}$ and $Q_{--}$ are Fredholm. The proof is  quite different from \cite{BarStrohmaier15} because of the non-compactness of $\Sigma$. In fact,  one of the main steps of the proof is showing the compactness of the ``off-diagonal" terms $Q_{+-}$ and $Q_{-+}$. In  \cite{BarStrohmaier15} it is done by showing that these operators are Fourier integral operators of negative order. On a compact manifold this implies compactness. 
On a non-compact manifold to establish compactness of an operator one also needs estimates on its ``behavior at infinity". To obtain such estimates we consider two compactly supported cut-off functions $\phi, \psi:\Sigma\to [0,1]$ such that support of $\phi$ is ``much bigger" than $K$ and support of $\psi$ is ``much bigger" than support of $\phi$. We then write 
\[
		Q_{-+}\ = \ \phi\circ Q_{-+}\circ \psi 
	\ + \ \phi\circ Q_{-+}\circ (1-\psi)
	\ + \ (1-\phi)\circ Q_{-+}\circ \psi
	\ + \ (1-\phi)\circ Q_{-+}\circ (1-\psi),
\]
and proof the compactness of each of the four terms in the right hand side separately. Using the finite propagation speed property of the wave equation we reduce the study of the first three terms to a study of the evolution operator on a compact manifold. 
The last term is supported on the complement of $[0,1]\times K$ where $\DD$ is a product.  We then show that this term is compact using an explicit computation of the restriction of $Q$ to the complement of $[0,1]\times K$.

Let $\DD_{APS}$ denote the operator $\DD$ with APS boundary conditions.  A verbatim repetition of the arguments in \cite[\S3]{BarStrohmaier15} shows that this operator is Fredholm and 
\begin{equation}\label{E:IindD=indQ}
	\ind\DD^+_{APS}\ = \ \ind Q_{--}.
\end{equation}

\subsection{The index formula}\label{SS:Iindexformula}
Let $\cDD$ be the ``Wick rotation" of $\DD$. This is an elliptic strongly Callias-type operator on $M$. We deform it to a new strongly Callias-type operator $\cDD^0$, which is product near $\pM$. Applying the deformation argument of \cite{Gilkey93} to the index formula of \cite{BrShi17,BrShi17b} we conclude that the APS index of $\cDD$ is given by 
\begin{multline}\label{E:Igeneralindex}
	\ind\cDD^+_{APS}\ = \ 
		\int_M\,\alpha_{AS}(\check{\DD}^+) 
		\ + \ \int_{\Sigma_1}  \TAS(\cDD^+,\cDD^{0+}) 
	\ - \ \int_{\Sigma_0}  \TAS(\cDD^+,\cDD^{0+})
	\\ + \ 	
	\frac{\eta(\AA_0,\AA_1)-\dim\ker\AA_0-\dim\ker\AA_1}{2},
\end{multline}
where $\alpha_{AS}(\cDD^+)$ is the Atiyah-Singer integrand,  $\TAS(\cDD^+,\cDD^{0+})$ is the transgression of $\alpha_{AS}(\cDD^+)$, and $\eta(\AA_0,\AA_1)$ is the {\em relative eta-invariant} introduced in 
\cite{BrShi17,BrShi17b}. 
Morally, the relative eta invariant $\eta(\AA_0,\AA_1)$ is the difference of the eta-invariants of $\AA_1$ and $\AA_0$, but the later invariants might not be defined in non-compact case. However, it is shown in \cite{BrShi17b} that, in many respects,  $\eta(\AA_0,\AA_1)$  behaves like it were the difference. In particular,  if $\AAA:= \{\AA_1^t\}_{0\le t\le 1}$ is a smooth family of strongly Callias-type operator, whose restriction to $M\backslash\big([0,1]\times K\big)$ is indepnedent of $t$. Then the spectral flow $\spf(\AAA)$ is well defined and 
\begin{equation}\notag
	2\spf(\AAA)\ = \ 
	\eta(\AA_1^1,\AA_0)\ - \ \eta(\AA_1^0,\AA_0)\ - \ 
	\int_0^1\,\big(\frac{d}{ds}\oeta(\AA_1^s,\AA_0)\big)\,ds,
\end{equation}
where $\oeta(\AA_1^s,\AA_0)\big)$ denote the reduction of $\eta(\AA_1^s,\AA_0)\big)$ modulo integers.
Using this formula and the arguments in Sections~4.1-4.2 of \cite{BarStrohmaier15} we show that 
\begin{equation}\notag
	\ind\cDD^+_{APS}\ = \ \spf(\AAA) \ - \ \dim\ker(\AA_1)
	\ = \ 
	\ind Q_{--}.
\end{equation}
Combining this equality with \eqref{E:IindD=indQ} we conclude that 
\[
	\ind\cDD^+\ = \ \ind\DD^+.
\]
In particular, $\ind\DD^+$ is given by the APS-type formula \eqref{E:Igeneralindex}.

\subsection*{Acknowledgments}
I would like to thank the Max Plank Institute for Mathematics in Bonn, where most of this work was conducted. I am also grateful to Christian B\"ar, Pengshuai Shi,  Matthias Lesch, Werner Ballmann, and Yafet Sanchez Sanchez for valuable discussions. 

\section{The setting}\label{S:setting}

In this section we introduce our main objects: a spatially non-compact globally hyperbolic  manifold $X$ and a Callias-type operator on it. When possible we use the notation of \cite{BarStrohmaier15}.

\subsection{A Dirac bundle over a globally hyperbolic manifold}\label{SS:glhyperbolic}

Let $\Sigma$ be a (possibly non-compact) odd-dimensional manifold and set $M:= [0,1]\times \Sigma$.  We endow $M$ with 
Lorentzian metric given by 
\begin{equation}\label{E:gM}
	\<\cdot,\cdot\>\ := \ - N^2\,dt^2\ + \ g^\Sigma_t, 
\end{equation}
where $g^\Sigma_t$ ($0\le t\le 1$) is a smooth family of complete Riemannian metrics on $\Sigma$ and $N=N(x,t)$ is a smooth  function, called the {\em lapse function}. 

The manifold $M$ is foliated by spacelike (i.e. Riemannian) hypersurfaces 
 $\Sigma_t:= \{t\}\times \Sigma$. We denote by $\nu$ the past-directed timelike vector field on $M$ with $\<\nu,\nu\>=-1$ which is perpendicular to all $\Sigma_t$. In coordinates, we have $\nu=-\frac1N\frac{\p}{\p t}$, where $N$ is the lapse function, cf.  \eqref{E:gM}.

\begin{definition}\label{D:Dirac bundle}
A (graded) {\em Dirac bundle} over $M$ is a graded vector bundle $E=E^+\oplus E^-$ endowed with 
\begin{enumerate}
\item a graded Clifford action $\gamma:TM\to \End(E)$, such that $\gamma(v)^2= -\<v,v\>$ and $\gamma(v):E^\pm\to E^\mp$ ($v\in TM$); we set $\beta:= \gamma(\nu)$ so that $\beta^2=1$;
\item 
a positive definite scalar product $(\cdot,\cdot)_E$ such that $E^+$ is orthogonal to $E^-$ and the indefinite inner product 
\begin{equation}\label{E:indefinnerproduct}
	\<\cdot,\cdot\>_E\ := \ (\cdot,\beta\cdot)_E
\end{equation}
satisfies 
\begin{equation}\label{E:gamma*}
	\big\<\gamma(v)e_1,e_2\big\>_E\ + \ \big\<e_1,\gamma(v)e_2\big\>_E
	\ = \ 0, \qquad e_1,e_2\in E, \ v\in TM;
\end{equation}

\item 
A connection $\n^E$ which preserves the grading  and satisfies the Leibniz rule 
\begin{equation}\label{E:Leibniz}
	\n^E_u\,\big(\gamma(v)e\big)\ = \ \gamma(\n^{LC}_uv)\,e\ + \ 
	\gamma(v)\,\n^E_uv,
	\qquad u,v\in TM, e\in E, 
\end{equation}
where $\n^{LC}$ is the Levi-Civita connection of the Lorenzian metric $\<\cdot,\cdot\>$. 
\end{enumerate} 
\end{definition}

Note that \eqref{E:gamma*} implies that
\begin{equation}\label{E:gammascalar}
	\big\<\gamma(v)e_1,\gamma(v)e_2\big\>_E\ =  \
	\<v,v\>\cdot \<e_1,e_2\>_E.
\end{equation}

If $\Sigma$ is a spin manifold, then so is $M$ and the bundle $SM$ of spinors over $M$ is naturally a Dirac bundle, cf. \cite{BarGauduchonMoroianu05}. More generally, if $W$ is a Hermitian bundle over $M$ endowed with a Hermitian connection, then $SM\otimes W$ is naturally a Dirac bundle. Any Dirac bundle locally looks like this.

\subsection{The Dirac operator}\label{SS:dirac}
Let  $D:C^\infty(M,E)\to C^\infty(M,E)$ be the Dirac operator associated to the connection $\n^{E}$. Locally, if $e_0,e_1,\ldots,e_n$ is an orthonormal frame (with respect to the Lorentzian metric \eqref{E:gM})  then 
\[
	D\ = \ \sum_{j=0}^n\,\epsilon_j\gamma_j\n^{E}_{e_j}, 
\]
where $\epsilon_j:=\<e_j,e_j\>=\pm1$. Then $D$ is odd with respect to the grading $E=E^+\oplus E^-$, i.e, has the form 
\[
	D\ = \ \begin{pmatrix}
	0&D^-\\
	D^+&0
	\end{pmatrix}.
\]

For a linear operator $L:C^\infty(M,E)\to C^\infty(M,E)$ we denote by $L^\dagger$ its formal adjoint  with respect to the indefinite inner product $\<\cdot,\cdot\>_E$. One readily sees that $D^\dagger= -D$, i.e, $(D^\pm)^\dagger= -D^\mp$.

\subsection{The restriction of $D$ to a hyperserface}\label{SS:restriction}
Let $E_t= E^+|_{\Sigma_t}$ ($0\le t\le 1$) denote the restriction of $E^+$ to $\Sigma_t$. We endow $E_t$ with the Clifford action 
\begin{equation}\label{E:gammat}
	\gamma_t(v)\ := \ i\beta\gamma(v), \qquad v\in T\Sigma_t.
\end{equation}
Then $E_t$ is an ungraded Dirac bundle over $\Sigma_t$. Let $A_t:C^\infty(\Sigma_t,E_t)\to C^\infty(\Sigma_t,E_t)$ denote the Dirac operator on $\Sigma_t$. By \cite[Eq. (3.6)]{BarGauduchonMoroianu05} (see also \cite[Eq. (3)]{BarStrohmaier15} and \cite[\S7.1-7.2]{BrMaschler17}) along $\Sigma_t$ we have
\begin{equation}\label{E:restrictiongeneral}
	D\ = \ -\beta\,\Big(\,\n^{E}_\nu+iA_t-\frac{n}{2}H_t\,\Big),
\end{equation}
where $H_t$ is the mean curvature of $\Sigma_t$ with respect to $\nu$. The reason for the term $\frac{n}{2}H_t$ is that the restriction of the Levi-Cevita connection to $\Sigma_t$ is not equal to the Levi-Civita connection of $\Sigma_t$, cf. \cite[\S7.1-7.2]{BrMaschler17} for more details.

\subsection{The Callias potential}\label{SS:Callias}

Let now $\Phi:E\to E$ be a self-adjoint bundle map. We denote by $\Phi_t$ the restriction of $\Phi$ to $E_t$. 

Consider the operator 
\begin{equation}\label{E:Callias+}
	\DD\ := \ D\ + \ i\beta\,\Phi:\,
	C^\infty(M,E)\ \to  \ C^\infty(M,E).
\end{equation}

Then along $\Sigma_t$ ($t\in [0,1]$) we have 
\begin{equation}\label{E:Calliasnear}
	\DD\ =  \ -\beta\,\Big(\,\n^{E}_\nu  + i\,\AA_t
	 - \frac{n}{2}H_t\,\Big)
\end{equation}
where 
\begin{equation}\label{E:CalliasA}
	\AA_t \ := \ A_t\ - \ \Phi_t
\end{equation}
is a Dirac-type operator on $\Sigma_t$.

\begin{definition}\label{D:Callias}
We say that $\AA_t$ is a {\em strongly Callias-type operator} if 
\begin{enumerate}
\item the ant-commutator $[\AA_t,\Phi_t]_+:= \AA_t\Phi_t+\Phi_t \AA_t$ is a zeroth order differential operator, i.e. a bundle map;
\item for any $R>0$, there exists a compact subset $K_R\subset M$ such that
\begin{equation}\label{E:strinvinfA}
	\Phi^2(x)\;-\;\big|[\AA_t,\Phi_t]_+(x)\big|\;\ge\;R
\end{equation}
for all $x\in M\setminus K_R$.  In this case, the compact set $K_R$ is called an \emph{R-essential support} of $\AA_t$.
\end{enumerate} 
\end{definition}

\begin{remark}\label{R:Psi anticommutes}
Condition (i) of Definition~\ref{D:Callias} is equivalent to the condition that $\Phi$ anticommutes with the Clifford multiplication: $\big[\gamma(v),\Phi\big]_+ = 0$, for all $\xi\in T^*M$.
\end{remark}

One readily sees that  a strongly Callias-type operator has a discrete  spectrum, \cite[\S3.10]{BrShi17}. 

\subsection{Assumption}\label{SS:assumption}
We now formulate the main assumptions under which we study the index of the operator $\DD$. 

\begin{enumerate}
\item[\textbf{(A1)}] 
There exists a compact set $K\subset \Sigma$ such that ${g^\Sigma_t}\big|_{{}_{\Sigma\backslash K}}$ is independent of $t$.
\item[\textbf{(A2)}]  
The  {\em lapse function} $N=N(x,t)$ is smooth and satisfies  $N(x,t)= 1$ for $x\not\in K$.
\item[\textbf{(A3)}] 
There is a fixed isomorphism 
\begin{equation}\label{E:isomofrestrictions}
	E|_{M\backslash\big([0,1]\times K\big)}\ \simeq \ 
	[0,1]\times E_0|_{\Sigma_0\backslash K}.
\end{equation}
Under this isomorphism the connection $\n^E$ is equal to the product of the connection on $E_0|_{\Sigma_0\backslash K}$ and the trivial connection along $[0,1]$. In other words, if we write a tangent vector to $M\backslash\big([0,1]\times K\big)\simeq [0,1]\times (\Sigma\backslash K)$ as $(a\nu,v)$ ($a\in \RR,\ v\in T\Sigma$), then 
\begin{equation}\label{E:restrictionnE}
	\n^E_{(a\nu,v)}|_{M\backslash\big([0,1]\times K\big)} \ = \
	-\frac{a}{N}\,\frac{\p}{\p t}\ + \ \n^{E_0}_v\big|_{\Sigma_0\backslash K}.
\end{equation}

\item[\textbf{(A4)}] 
The restriction of $\Phi_t$ to  $M\backslash\big([0,1]\times K\big)$ is independent of $t$

\item[\textbf{(A5)}] 
$\AA_0$ is a strongly Callias-type operator in the sense of Definition~\ref{D:Callias}.
\end{enumerate}

Note that Assumption (A3) and the Leibniz rule \eqref{E:Leibniz} imply that the restriction of the Clifford action $\gamma$ to $M\backslash\big([0,1]\times K\big)$ is independent of\/ $t\in [0,1]$. Since the inner product $\<\cdot,\cdot\>_E$ is preserved by $\n^E$, its restriction to  $M\backslash\big([0,1]\times K\big)$ is also independent of $t$.

It follows from Assumptions~(A1)-(A3) that the restriction of $\AA_t$ to $\Sigma_t\backslash\big(\{t\}\times K\big)$ is independent of $t$. Hence, Assumption~(A4) implies that $\AA_t$ is a strongly Callias-type operator for all $t\in [0,1]$.

\subsection{Restriction to $M\backslash([0,1]\times K)$}\label{SS:restrictionM-K}
We note that, by Assumptions~(A1)-(A3) of Section~\ref{SS:assumption},  the restriction of $\DD$ to $M\backslash\big([0,1]\times K\big)$ is equal to 
\begin{equation}\label{E:restrictionDD}
	\DD|_{M\backslash\big([0,1]\times K\big)}\ = \ 
	-\beta\, \Big(\,-\frac{\p}{\p t}  + i\,\AA_0\,\Big).
\end{equation}

\subsection{The scale of Sobolev spaces}\label{SS:Sobolev}

We recall the definition of Sobolev spaces $H^s_{\AA_t}(\Sigma_t,E_t)$ of sections over $\Sigma_t$ which depend on the  operator $\AA_t$, cf. \cite[\S3.13]{BrShi17}.

\begin{definition}\label{D:Sobolev}
Set
\[
	C_{\AA_t}^\infty(\Sigma_t,E_t)	\;:=\;
	\Big\{\,\bfu\in C^\infty(\Sigma_t,E_t):\,
	\big\|(\id+\AA_t^2)^{s/2}\bfu\big\|_{L^2(\Sigma_t,E_t)}^2<+\infty\mbox{ for all }s\in\RR\,\Big\}.
\]
For all $s\in\RR$ we define the \emph{Sobolev $H_{\AA_t}^s$-norm} on $C_{\AA_t}^\infty(\Sigma_t,E_t)$ by
\begin{equation}\label{E:Sobnorm}
	\|\bfu\|_{\HAs{t}}^2\;:=\;
		\big\|(\id+\AA_{t}^2)^{s/2}\bfu\big\|_{L^2(\Sigma_t,E_t)}^2.
\end{equation}
The Sobolev space $\HAs{t}$ is defined to be the completion of $C_{\AA_{t}}^\infty(\Sigma_t,E_t)$ with respect to this norm.
\end{definition}

\section{The Cauchy problem and the evolution operator}\label{S:evolution}

In this section we show that the Cauchy problem for the strongly Callias-type operator $\DD$ is well-posed in certain spaces closely related to the Sobolev spaces $\HAs{t}$. We then construct the {\em evolution operator}\/ $Q:\HAs{0}\to \HAs{1}$ for the wave equation $\DD\bfu=0$ and discuss its basic properties. This generalizes the results of 
\cite[\S2]{BarStrohmaier15} to our non-compact situation.

\subsection{The compact case}\label{SS:compact}

Fix $t_0\in [0,1]$ and let $\bfu= \E_{t_0}(\bfu_0;\bff)$ denote the solution of the non-homogeneous Cauchy problem
\begin{equation}\label{E:wave}
	\begin{aligned}
		\DD\bfu\ &= \ \bff;\\
		\bfu(t_0,x)\ &= \ \bfu_{t_0}(x). 
	\end{aligned}
\end{equation}
In particular, in the case $t_0=0$ and $\bff=0$ we define the {\em evolution operator} $Q$ as the map from sections over $\Sigma_{0}$ to sections  over $\Sigma_1$, defined by 
\[
	Q:\,\bfu_0 \to \E_0(\bfu_0;0)_{|_{\Sigma_1}}.
\] 
To make this definition rigorous one needs to define a space of sections over $M$ such that the wave equation \eqref{E:wave} has a unique solution in this space. It is also desirable to show that the operator $Q$ is continuous. 

A rigorous construction of the evolution operator in the case when $\Sigma$\/ is compact is given by B\"ar and  Strohmaier in \cite[\S2]{BarStrohmaier15}. In this paper the authors considered the collection $\HAs{t}$ as an infinite dimensional vector bundle over $[0,1]$ and defined the space 
$FE^s(M,D)$ of {\em finite $s$-energy sections} to be the completion of the space of continuous sections of this bundle with respect to  a certain norm. 

Let
\begin{equation}\label{E:BSsolution}
	FE^s(M,\ker D)\ := \ 
	\big\{\,\bfu:\,  \bfu(t,\cdot)\in \HAs{t},\, D\bfu=0\,\big\}.
\end{equation}
(note that, if $\Sigma$ is compact, then $\HAs{t}$ coincides with  the usual Sobolev space $H^s(\Sigma_t,E_t)$). 

B\"ar and  Strohmaier proved the following version of well-posedness  of the Cauchy problem for the wave equation: if $\Sigma$ is compact, then the restriction map
\begin{equation}\label{E:BSres}
	\res_t:\, FE^s(M,\ker D) \ \to \ \HAs{t}
\end{equation}
is an isomorphism of topological vector spaces. This allows to define the evolution operator as 
\begin{equation}\label{E:BSevolution}
	Q\ := \ \res_1\circ \res_0^{-1}:\,
	\HAs{0}\ \to \ 
	\HAs{1}.
\end{equation}

\subsection{Sketch of the construction in non-compact situation}\label{SS:sketchevolution}

Let us first briefly explain where our construction of $Q$ differs from the compact case and sketch our strategy of treating the problems at ``infinity".

To extend the  argument  of the previous subsection to our non-compact situation we consider the collection of spaces $\HAs{t}$ as an infinite dimensional bundle over $[0,1]$ and define the space $FE^s(M,\DD)$ to be  a completion of the space of continuous sections of this bundle. We define the space $FE^s(M,\ker\DD)$ as in \eqref{E:BSsolution}. The main result of this section is the {\em well-posedness of the Cauchy problem \eqref{E:wave} in these spaces}, which means that, for each $t_0\in [0,1]$,  the map 
\begin{equation}\label{E:sketchwell}
	\res_{t_0}\oplus \DD:\,FE^s(M,\DD)\ \to \ 
	\HAs{t_0}\oplus L^2\big(\,[0,1],\HAs{t}\,\big)
\end{equation}
is an isomorphism of Banach spaces. In particular, 
\begin{equation}\label{E:res}
	\res_{t_0}:\, FE^s(M,\ker\DD) \ \to \ \HAs{t_0}
\end{equation}
is an isomorphism of Banach spaces and we can define the evolution operator by \eqref{E:BSevolution}. 

To prove that \eqref{E:sketchwell} is an isomorphism, we construct the inverse $\E_{t_0}$ of this map. So $\E_{t_0}(\bfu_{t_0};\bff)$ is the solution of \eqref{E:wave}.

Consider a large compact set $K'\subset M$. In particular, we assume that $K'$ contains the set $K$ defined in (A1) of Section~\ref{SS:assumption}. Let $\phi$ be a compactly supported function whose restriction to $K'$ is equal to 1. Fix $t_0\in [0,1]$. Given $\bfu_{t_0}\in \HAs{t}$, $\bff\in L^2\big(\,[0,1],\HAs{t}\,\big)$,  we decompose them as $\bfu_{t_0}= \phi\bfu_{t_0}+(1-\phi)\bfu_{t_0}$,  $\bff= \phi\bfu+(1-\phi)\bfu$ and construct separately $\E_{t_0}(\phi\bfu_{t_0},\phi\bff)$ and $\E_{t_0}((1-\phi)\bfu_{t_0},(1-\phi)\bff)$.

Since $\phi\bfu_{t_0}$ and $\phi\bff$ are compactly supported it follows easily from the finite propagation speed of the solutions of the wave equation and a result of \cite[\S2]{BarStrohmaier15} that there is a unique compactly supported section $\bfu'= \E_{t_0}(\phi\bfu_{t_0},\phi\bff)\in FE^s(M,\DD)$ such that $\res_{t_0}(\bfu')= \phi\bfu_{t_0}$ and $\DD\bfu'= \phi\bff$. 

The supports of $(1-\phi)\bfu_{t_0}$ and $(1-\phi)\bff$ are outside of $K$. 
Suppose 
\[
	\bfu''\ = \
	 \E_{t_0}\big((1-\phi)\bfu_{t_0},(1-\phi)\bff\big)\in FE^s(M,\DD)
\] 
is a solution of the wave equation. If the set $K'$ is large enough, then, from  the finite propagation speed property, we conclude that $\bfu''$ is supported outside of the set $[0,1]\times K$, where $K$ is the compact set defined in (A1) of Section~\ref{SS:assumption}. Recall that all our structures are product outside of $[0,1]\times K$. Using this fact, one checks that
\[
	\bfu''(t,x) \ = \ e^{i(t-t_0)\AA_0}(1-\phi)\bfu_{t_0}(x)
	\ - \ \beta\,
	\int_{t_0}^t\,e^{i\AA_0(t-s)}\,(1-\phi)\bff(s)\,ds.
\]

We now define the extension map 
\begin{equation}\label{E:sketchextension}
		\E_{t_0}:\HAs{t}\oplus L^2\big(\,[0,1],\HAs{t}\,\big)
		\ \to \ FE^s(M,\DD)
\end{equation}
by $\E_{t_0}(\bfu_t;\bff)= \bfu'+\bfu''$. One easily checks that this map is independent of the choice of the cut-off function $\phi$. We prove that this map is the inverse of  $\res_{t_0}\oplus\DD$ and, hence, is  an isomorphism of Banach spaces. The evolution operator is defined by \eqref{E:BSres}.

\subsection{The propagation speed}\label{SS:propagationspeed}

Many of the arguments in this paper are based on the following {\em finite propagation speed} property of the wave equation, cf., for example, \cite[Theorem~7.9]{Alinhac09}.

\begin{definition}\label{D:minmetric}
Define a new norm  on the tangent space $T\Sigma$ to $\Sigma$ by 
\begin{equation}\label{E:gmin}
	\|u\|_{\max}^2\ := \ 
	\frac{\displaystyle\max_{0\le t\le 1} g^\Sigma_t(u,u)}
	{\displaystyle\min_{(t,x)\in M} N(t,x)}, 
	\qquad u\in T_x\Sigma, 	
\end{equation}
where $N(t,x)$ is the lapse function, cf. \eqref{E:gM}.
For $x,y\in \Sigma$ we denote by $\dist(x,y)$ the distance between $x$ and $y$ defined by the Finsler  metric associated to this norm. 
\end{definition}
\begin{proposition}[Finite propagation speed]\label{P:finiteps}
Suppose $s>1/2$ and $u\in H^s_{loc}(M,E)$ is a solution of the equation $\DD \bfu=\bff$. For\/ $t_0\in [0,1]$, we have 
\begin{equation}\label{E:finiteps}
	\supp \bfu \ \subset \ 
	\big\{\, (t,x):\, 
	\dist\big(\,x,\supp \bfu_{t_0}\cup \supp \bff\,\big
		)\le |t-t_0|\, \big\}.
\end{equation}
\end{proposition}

\subsection{The space of finite $s$-energy sections}\label{SS:finiteenergy} 

We view the family 
\[
	\HAs{t}\ :=\ \HAsp{t}, \qquad t\in [0,1]),
\] 
as an infinite dimensional vector bundle over $[0,1]$. The space of continuous sections of this bundle is called the space of {\em  finite $s$-energy sections}  and is denoted by $FE^s(M,E)$. We endow $FE^s(M,E)$ with the norm 
\begin{equation}\label{E:normFEs}
	\|\bfu\|_{FE^s}\ := \ \max_{t\in[0,1]}\,\|\bfu(t)\|_{\HAs{t}}.
\end{equation}
We also consider the space of $L^2\big([0,1],\HAs{t}\big)$ of sections of $\HAs{t}$ with finite $L^2$-norm
\begin{equation}\label{E:normFEsL2}
	\|\bfu\|^2_{L^2,\HAs{t}}\ := \ 
	\int_0^1\,\big\|(N\bfu)|_{\Sigma_t}\big\|^2\,dt.
\end{equation}

The operator $\DD$ acts on $FE^s(M,E)$ by (cf. \eqref{E:Calliasnear}) 
\begin{equation}\label{E:DDonFE}
	\DD\bfu\ := \  -\beta\,\Big(\,
	-\frac1N\,\frac{d\bfu}{dt}  + i\,\AA_t\,\bfu(t)
	 - \frac{n}{2}H_t\,\bfu(t)\,\Big)	
\end{equation}

We define the space $FE^s(M,\DD)$ as the completion of $FE^s(M,E)$ with respect to the norm
\begin{equation}\label{E:normFED}
	\|\bfu\|_{FE^s,\DD}^2\ := \  
	\|\bfu\|_{FE^s}^2\ +\ \|\DD\bfu\|^2_{L^2,FE^s}.
\end{equation}
Notice, that for each $t\in [0,1]$ there is a well-defined continuous map
\begin{equation}\label{E:rest}
	\res_t:\,FE^s(M,\DD) \to \ \HAs{t}.
\end{equation}
We set
\begin{equation}\label{E:FEkerD}
	FE^s(M,\ker \DD)\ := \ 
	\big\{\,\bfu\in FE^s(M,\DD):\, \DD\bfu=0\,\big\}.
\end{equation}

\subsection{Well-posedness of the inhomogeneous Cauchy problem}\label{SS:wellposed}
The following theorem extends  Theorem~2.1 of \cite{BarStrohmaier15} to our non-compact situation:

\begin{theorem}\label{T:restriction}
For each $t_0\in [0,1]$ the restriction map 
\begin{equation}\label{E:Trestriction}
	\res_{t_0}\oplus \DD:\, FE^s(M,\DD)\ \to \ 
	\HAs{t_0}\oplus L^2\big([0,1],\HAs{t}\big),
\end{equation}
is an isomorphism of Banach spaces. 
\end{theorem}

The proof of the theorem occupies Sections~\ref{SS:cutoff}-\ref{SS:prrestriction}. First, we mention the following direct corollary:

\begin{corollary}\label{C:restriction}
For each $t_0\in [0,1]$ the restriction map 
\begin{equation}\label{E:Crestriction}
	\res_{t_0}:\, FE^s(M,\ker\DD)\ \to \  \HAs{t_0},
\end{equation}
is an isomorphism of Banach spaces. 
\end{corollary}
\subsection{Cut-off functions}\label{SS:cutoff}
Let $K$ be the compact set defined in (A1) of Section~\ref{SS:assumption}. Let $K'\Supset K$ be a compact subset of $\Sigma$ such that 
\begin{equation}\label{E:disttoK}
	\dist(K,M\backslash K') \ > \ 2. 
\end{equation}

Let $\phi:\Sigma\to [0,1]$ be a smooth compactly supported function such that $\phi|_{K'}=1$. 

We choose a compact set  $K''\Supset \supp\phi$ such that 
\begin{equation}\label{E:disttosuppphi}
	\dist(\supp\phi,M\backslash K'') \ > \ 2. 
\end{equation}
Let $\psi:\Sigma\to [0,1]$ be a smooth compactly  supported function with $\psi|_{K''}=1$.

Finally we choose a compact set $K'''\subset \Sigma$ such that
\begin{equation}\label{E:K'''}
		\dist(\supp\psi,M\backslash K''') \ > \ 2.
\end{equation}

Clearly, 
\begin{equation}\label{E:phi+1-phi}
	\E_{t_0}(\bfu_{t_0},\bff) \ = \ \E_{t_0}(\phi\bfu_{t_0},\phi\bff)
	\ + \ \E_{t_0}\big((1-\phi)\bfu_{t_0},(1-\phi)\bff\big).
\end{equation}
We  construct each of the summands in the right hand side separately. 

\subsection{The case of compact support}\label{SS:extcomp}
Let $K'$, $K''$, and $K'''$ be as in the previous subsection.

\begin{lemma}\label{L:compactsupport}
Let $\bfv_{t_0}\in \HAs{t}$, $\bfg\in L^2\big([0,1],\HAs{t}\big)$ be sections with  support in $K''$.  Then there is a unique solution $\bfv=\E_{t_0}(\bfv_{t_0},\bfg)\in FE^s(M,\DD)$ of the wave equation 
\begin{equation}\label{E:wavecompact}
		\DD\bfv\ = \ \bfg;\qquad
		\bfv(t_0,x)\ = \ \bfv_{t_0}(x). 
\end{equation}
Moreover, $\supp \E_{t_0}(\bfv_{t_0},\bfg)\subset [0,1]\times K'''$.
\end{lemma}

\begin{proof}
Choose a compact manifold $\tilde{\Sigma}$ which contains $K'''$ and consider the product $\tilde{M}:= [0,1]\times\tilde{\Sigma}$. Let $\widetilde{\<\cdot,\cdot\>}$ be a Lorenzian metric on $\tilde{M}$ whose restriction to $[0,1]\times K''$ coincides with $\<\cdot,\cdot\>$. Let $\tilde{E}$ be a Dirac bundle over $\tilde{M}$ whose restriction to $[0,1]\times K''$ coincides with $E$. 
We then have a Dirac operator $\tilde{D}$ on $\tilde{E}$, whose restriction to $[0,1]\times K''$ coincides with $D$. Fix a self-adjoint bundle map $\tilde{\Phi}:\tilde{E}\to \tilde{E}$ whose restriction to $[0,1]\times K''$ coincides with $\Phi$ and consider the operator $\tilde{\DD}:=\tilde{D}+i\beta\tilde{\Phi}$.  Let $FE^s(\tilde{M},\tilde{\DD})$ denote the space of finite $s$-energy sections on $\tilde{M}$. From Theorem~2.1 of \cite{BarStrohmaier15} we conclude that the map 
\begin{equation}\label{E:restrictiontilde}
	\widetilde{\res}_{t_0}\oplus\tilde{\DD}:\, 
	FE^s(\tilde{M},\tilde{\DD}) \ \to \ 
	H^s(\tilde{\Sigma}_{t_0},\tilde{E}_{t_0})
	\oplus
	L^2\big(\,[0,1],\HAs{t}\,\big)
\end{equation}
is an isomorphism of Banach spaces. Denote by  $\tilde{\E}_{t_0}$ its inverse. It follows from the finite propagation speed, Proposition~\ref{P:finiteps}, that $\supp \tilde{\E}_{t_0}(\bfv_{t_0},\bfg)\subset K'''$ and that $\tilde{\E}_{t_0}(\bfv_{t_0},\bfg)$ does not depend on the choices of $\tilde{\Sigma}$, $\widetilde{\<\cdot,\cdot\>}$,  $\tilde{E}$, and $\tilde{\Phi}$. 

For $X\subset \tilde{\Sigma}$ we denote
\begin{equation}\label{E:HX}
	H^s(X)\ : = \ 
	\big\{\,\bfv_{t_0}\in 
	H^s(\tilde{\Sigma}_{t_0},\tilde{E}_{t_0})
	:\, \supp\bfv_0\subset X\,\big\},
\end{equation}
and
\begin{equation}\label{E:FEX}
	FE^s(X)\ : = \ 
	\big\{\,\bfv\in FE^s(\tilde{M},\tilde{\DD})
	:\, \supp\bfv\subset [0,1]\times X\,\big\}.
\end{equation}
Then 
\begin{equation}\label{E:tilderesext1}
	\begin{aligned}
	&\widetilde{\res}_{t_0}\oplus \DD:\, FE^s(K'')\ \to \ 
	H^s(K'')\oplus L^2\big(\,[0,1],H^s(K''')\,\big),
	\\
	&\tilde{\E}_{t_0}:\, H^s(K'') \oplus 
	   L^2\big([0,1],H^s(K''')\big)\   \to \  FE^s(K'''),
	\end{aligned}
\end{equation}
and
\begin{equation}\label{E:tilderesext}
	\begin{aligned}
	\widetilde{\res}_{t_0}\circ \tilde{\E}_{t_0}\,(\bfv_{t_0},\bfg)
	\ &= \ \bfv_{t_0}, 
	\qquad \text{if} \quad \supp \bfv_{t_0}\subset K'';\ \ 
	\supp \bfg\subset K'';\\
	\tilde{\E}_{t_0}\circ 
	\big(\,\widetilde{\res}\oplus \tilde{\DD}\,\big)\,(\bfv)
	\ &= \ \bfv, 
	\qquad \text{if} \quad \supp \bfv\subset [0,1]\times K''.
  \end{aligned}
\end{equation}

We view $H^s(K'')$ as a subset of $\HAs{t_0}$ and $FE^s(K''),\ FE^s(K''')$ as subsets of $FE^s(M,\DD)$. Then
\[
	\E_{t_0}(\bfv_{t_0},\bfg)\ := \ \tilde{\E}_{t_0}(\bfv_{t_0},\bfg)
	\ \subset \ FE^s(K''')\ \subset \ FE^s(M,\DD)
\]
satisfies the wave equation \eqref{E:wavecompact}. The uniqueness of this solution is a direct consequence of the abstract uniqueness theorem for solutions of differential equations in Banach spaces, \cite[Theorem~1]{Kato56}.
\end{proof}

\subsection{Solution in a neighborhood of infinity}\label{SS:nearinfty}
Let 
\[
	\bfw_{t_0}\in \HAs{t},  \qquad
	\bfh\in L^2\big([0,1],\HAs{t}\big)
\] 
be sections with  support in $M\backslash{}K'$.

\begin{lemma}\label{L:nearinfty}
The unique solution $\bfw=\E_{t_0}(\bfv_{t_0},\bfg)\in FE^s(M,\DD)$ of the wave equation 
\begin{equation}\label{E:waveinfty}
		\DD\bfw\ = \ \bfh;\qquad
		\bfw(t_0,x)\ = \ \bfw_{t_0}(x)
\end{equation}
is given by Duhamel's formula
\begin{equation}\label{E:dbtildeE}
	\bfw \ = \ \E_{t_0}(\bfw_{t_0},\bfh) \ := \ 
	e^{i(t-t_0)\AA_0}\,\bfw_{t_0}
	\ - \ \beta\,\int_{t_0}^t\,e^{i\AA_0(t-s)}\,\bfh(s)\,ds.
\end{equation} 
Moreover, $\supp \E_{t_0}(\bfw_{t_0},\bfh)\subset [0,1]\times M\backslash{K}$.
\end{lemma}

\begin{proof}
Recall that the restriction of $\DD$ to $M\backslash{}K \supset M\backslash{}K'$ is given by \eqref{E:restrictionDD}. By Hille-Yosida theorem, \cite[Theorem~X.47a]{RiSi2}, $\AA_{t_0}$ generates a strongly continuous contraction semigroup $e^{i\AA_0t}$ and the solution of the non-homogeneous ``time independent" wave equation 
\begin{equation}\label{E:wavenotime}
	\begin{aligned}
		-\beta\, \Big(\,-\frac{\p}{\p t}  + i\,&\AA_0\,\Big)\bfw\ 
		= \ \bfh;\\
		\bfw(t_0,x)\ &= \ \bfw_{t_0}
	\end{aligned}
\end{equation} 
is given by \eqref{E:dbtildeE}. 

By the finite propagation speed property  the sections $e^{i(t-t_0)\AA_0}\,\bfw_{t_0}$ and $e^{i\AA_0(t-s)}\,\bfh$ are supported in $M\backslash{}K$. Hence, it is also a solution of \eqref{E:waveinfty}. As in the proof of Lemma~\ref{L:compactsupport}, the uniqueness of this solution follows from \cite[Theorem~1]{Kato56}. It remains to show that $\bfw\in FE^s(M,\DD)$. 

Since $e^{i(t-t_0)\AA_0}$ is a strongly continuous family of operators on $\HAs{0}$, the functions 
\begin{equation}\label{E:functions}
	\begin{aligned}
		t&\mapsto e^{i\AA_0(t-t_0)}(1-\phi)\bfu_{t_0}
		\ \in \ \HAs{0},\\
		t&\mapsto \beta\,\int_{t_0}^t\,e^{i\AA_0(t-s)}\,\bfh(s)\,ds
		\ \in \ \HAs{0}
	\end{aligned}
\end{equation}
are norm-continuous. By the finite propagation speed property, the support of $e^{i(t-t_0)\AA_0}(1-\phi)\bfw_{t_0}$ is a subset of $M\backslash{K}$. Since the restrictions of the spaces $\HAs{t}$ to $M\backslash{K}$ are isometric, \eqref{E:functions} are also  continuous when viewed as sections of the bundle $\HAs{t}$. Hence, the right hand side of \eqref{E:dbtildeE} has finite $FE^s$-norm.

Since $\DD e^{i(t-t_0)\AA_0}(1-\phi)\bfw_{t_0}=0$, and 
\[
	\DD\,\Big[\beta\,\int_{t_0}^t\,e^{i\AA_0(t-s)}\,\bfh(s)\,ds\,\Big]
	\ = \ \bfh\ \in \ L^2\big([0,1],\HAs{t}\big),
\]
we see that the $L^2$-norm of $\DD\E_{t_0}(\bfw_{t_0},\bfh)$ is also finite. Hence,

\[
	\big\|\, \E_{t_0}(\bfw_{t_0},\bfh)\,\big\|_{FE^s,\DD}
	\ <\ \infty.
\]
\end{proof}

\subsection{Proof of Theorem~\ref{T:restriction}}\label{SS:prrestriction}
Let $\bfu_{t_0}\in \HAs{t_0}$, $\bff\in L^2\big([0,1],\HAs{t}\big)$. Set 
\begin{multline}\label{E:extension}
	\E_{t_0}(\bfu_{t_0},\bff)\ := \  
	\E_{t_0}(\phi\bfu_{t_0},\phi\bff)\ + \ 	
	\E_{t_0}\big((1-\phi)\bfu_{t_0},(1-\phi)\bff\big)
	\\ = \ 	
	\tilde{\E}_{t_0}(\phi\bfu_{t_0},\phi\bff)
	\ + \ 
	e^{i(t-t_0)\AA}\,(1-\phi)\,\bfu_{t_0}(x)
	\ - \ \beta\,\int_{t_0}^t\,e^{i\AA_0(t-s)}\,(1-\phi)\bff(s)\,ds.
\end{multline}
Clearly, $\E_{t_0}(\bfu_{t_0},\bff)$ satisfies the wave equation. From Lemmas~\ref{L:compactsupport} and \ref{L:nearinfty} we conclude that $\E_{t_0}(\bfu_{t_0},\bff)\in FE^s(M,\DD)$. Thus $\E:\HAs{t}\to FE^s(M,\ker\DD)$ is the inverse of $\res_{t_0}\oplus\DD$. In particular, $\res_{t_0}\oplus\DD:FE^s(M,\DD)\to \HAs{t_0}\oplus L^2\big([0,1],\HAs{t}\big)$ is a bijection. Since by construction $\res_t$ is a bounded linear map, the theorem follows from the Bounded Inverse Theorem, \cite[Theorem~III.11]{ReSi1}.\hfill$\square$

\subsection{The evolution operator}\label{SS:evolutionoperator}
Let $\E'_{t_0}:\HAs{t_0}\to FE^s(M,\ker\DD)$ denote the inverse of the isomorphism $\res_{t_0}$, cf. Corollary~\ref{C:restriction}. The isomorphism 
\begin{equation}\label{E:evolutionQ}
	Q:= \res_1\circ\, \E'_0:\,\HAs{0}\ \to \ \HAs{1},
\end{equation}
is called {\em the evolution operator}.

\begin{proposition}\label{P:unitarity}
For $s=0$ the operator $Q:L^2(\Sigma_0,E_0)\to L^2(\Sigma_1,E_1)$ is unitary. 
\end{proposition}
\begin{proof}
The proof is similar to the proof of Lemma~2.4 in \cite{BarStrohmaier15}. Since the space $C^\infty_c(\Sigma_0,E_0)$ is dense in $L^2(\Sigma_0,E_0)$, it is enough to check that 
\[
	\|Q\bfu_0\|_{L^2(\Sigma_1,E_1)} \ = \ 
	\|\bfu_0\|_{L^2(\Sigma_0,E_0)}
\]
for smooth compactly supported sections $\bfu_0$. Let $\bfu_0$ be such section. Then it belongs to $\HAs{0}$ for all $s$. By Theorem~\ref{T:restriction} there is a unique $\bfu\in FE^s(M,\ker\DD)$ whose restriction to $\Sigma_0$ is equal to $\bfu_0$. By the finite propagation speed property the support of $\bfu$ is compact. Since $\DD\bfu=0$, using equation \eqref{E:DDonFE} we conclude that 
\[
	\frac{d\bfu}{dt} \ = \  
	N\,\big(\,-i\,\AA_t\bfu- \frac{n}{2}H_t\,\big)\,\bfu.
\]
Hence, $\bfu$ is a $C^1$ section of $H^{s-1}_{\AA_t}(\Sigma_t,E_t)$.

Fix $s> \frac{n}2+2$. Then, by Sobolev embedding theorem, $H^{s-1}_{\AA_t}(\Sigma_t,E_t)\subset C^1(\Sigma_t,E_t)$. We conclude that $\bfu\in C^1_c(M,E)$.  Using the Green's formula (cf., for example, \cite[Eq.~(2)]{BarStrohmaier15}) we now obtain
\begin{multline}\notag
	0\ = \ \int_M\,\Big[\, (\DD\bfu,\bfu)+(\bfu,\DD\bfu)\,\Big]\,dV
	\\ = \ 
	\int_{\Sigma_1}\,(\beta Q\bfu_0,Q\bfu_0)\,dA\ - \ 
	\int_{\Sigma_0}\,(\beta\bfu_0,\bfu_0)\,dA
	\\ = \ \|Q\bfu_0\|_{L^2(\Sigma_1,E_1)} \ -\ 
	\|\bfu_0\|_{L^2(\Sigma_0,E_0)}.
\end{multline}
\end{proof}

\section{Properties of the evolution operator}\label{S:propertiesevolution}

In this section we show that most of the properties of the evolution operator $Q$ established in \cite{BarStrohmaier15} remain valid in our non-compact setting. Following \cite{BarStrohmaier15} we decompose the space of $L^2$-sections over $\Sigma_t$ into the direct sum of the spectral subspaces of $\AA_t$ corresponding to positive and negative parts of the spectrum. We write $Q$ as a matrix 
\begin{equation}\label{E:Qmatrix}
	Q\ = \ \begin{pmatrix}
	Q_{++}&Q_{+-}\\Q_{-+}&Q_{--}
	\end{pmatrix}
\end{equation}
with respect to this decomposition. One of the main result of this section is that the operators $Q_{++}$ and $Q_{--}$ are Fredholm. Our proofs are quite different from \cite{BarStrohmaier15} because of the non-compactness of $\Sigma$. In fact,  one of the main steps of the proof is showing the compactness of the ``off-diagonal" terms $Q_{+-}$ and $Q_{-+}$. In  \cite{BarStrohmaier15} it is done by showing that these operators are Fourier integral operators of negative order. On compact manifold this implies compactness. On non-compact manifold to establish compactness of an operator one also need to obtain estimates on its ``behavior at infinity". Most of this section is devoted to such estimates. 

\subsection{The spectral subspaces}\label{SS:spectralspaces}

For $I\subset \RR$ we denote by $\LSI{t}{I}\subset \LS{t}$
the spectral subspace of $\AA_t$ corresponding to the eigenvalues in $I$. The orthogonal projection $P_I^t:\LS{t}\to \LSI{t}{I}$ is call the {\em spectral projection} of $\AA_t$ corresponding to $I$.

We have $L^2$-orthogonal splittings
\begin{equation}\label{E:spdecomposition}
	\begin{aligned}
	\LS{0}\ &= \ \LSI{0}{[0,\infty)}\oplus \LSI{0}{(-\infty,0)};\\
	\LS{1}\ &= \ \LSI{1}{(0,\infty)}\oplus \LSI{0}{(-\infty,0]},
	\end{aligned}
\end{equation}
and write $Q$ as a $2\times2$-matrix \eqref{E:Qmatrix} with respect to this decomposition. Thus 
\[
	Q_{++}\ = \ P_{(0,\infty)}^1\circ Q_{\,|_{\LSI{0}{[0,\infty)}}},
\] 
etc. 

As a consequence of the unitarity of $Q$ (Proposition~\ref{P:unitarity}) we obtain the following
\begin{lemma}\label{L:Qpm}
The operator $Q_{+-}$ restricts to an isomorphism $\ker{}Q_{--}\to \ker{}Q_{++}$. Similarly, the operator $Q_{-+}$ restricts to an isomorphism $\ker{}Q_{++}\to \ker{}Q_{--}$.
\end{lemma}
\begin{proof}
The proof is a verbatim repetition of the proof of Lemma~2.5 of \cite{BarStrohmaier15}.
\end{proof}

\begin{lemma}\label{L:Qpms}
For every $s\ge0$, we have
\begin{equation}\label{E:HcomHloc}
 \begin{aligned}
	Q_{+-}\circ P_{(-\infty,0)}^0&:\,
	 H^s_{\comp}(\Sigma_{t_0},SM_0\otimes E_0)\ \to \  
	 H^{s+1}_{\loc}(\Sigma_{1},SM_1\otimes E_1),\\
	Q_{-+}\circ P_{[0,\infty)}^0&:\,
	 H^s_{\comp}(\Sigma_{t_0},SM_0\otimes E_0)\ \to \  
	 H^{s+1}_{\loc}(\Sigma_{1},SM_1\otimes E_1).
 \end{aligned}
\end{equation}
\end{lemma}
\begin{proof}
It is shown in the proof of Lemma~2.6 of \cite{BarStrohmaier15}  that $Q_{+-}\circ P_{(-\infty,0)}^0$ and $Q_{-+}\circ P_{[0,\infty)}^0$ are Fourier integral operators of order -1 whose canonical relation is a canonical graph.   The statement of the lemma follows from the mapping property of Fourier integral operators, \cite[Corollary~4.4.5]{Duistermaat96FIO}.
\end{proof}

In exactly the same way as in \cite[Corollary~2.7]{BarStrohmaier15} we obtain the following corollary

\begin{corollary}\label{C:regularity}
The kernels of $Q_{++}$ and $Q_{--}$ consist of smooth sections. 
\end{corollary}

\subsection{Compactness of the off-diagonal terms}\label{SS:compactness}
We are now formulate the main result of this section:
\begin{proposition}\label{P:compactness}
The operators $Q_{+-}$ and $Q_{-+}$ are compact. 
\end{proposition}

Before proving the proposition let us mention the following important corollary:

\begin{corollary}\label{C:Fredholm}
The operators $Q_{++}$ and $Q_{--}$ are Fredholm and 
\begin{equation}\label{E:Fredholm}
	\ind Q_{++}\ = \ - \ind Q_{--}. 
\end{equation}
\end{corollary}
\begin{proof}
By Proposition~\ref{P:unitarity}, $Q$ is a unitary operator. Hence, it is Fredholm with index 0. By Proposition~\ref{P:compactness}, 
\[
	Q\ - \ \begin{pmatrix}
	Q_{++}&0\\0&Q_{--}
	\end{pmatrix}
\]
is a compact operator. Hence,
\[
	0 \ = \ \ind Q \ = \ \ind \begin{pmatrix}
	Q_{++}&0\\0&Q_{--}
	\end{pmatrix}
	\ = \ \ind Q_{++}\ + \ \ind Q_{--}.
\]
\end{proof}

\subsection{Sketch of the proof of Proposition~\ref{P:compactness}}\label{SS:sketchcompactness}
It is enough to prove compactness of $Q_{-+}$. The proof for $Q_{+-}$ is analogous. 

If $\Sigma$ is a closed manifold, the compactness of $Q_{-+}$ follows from Lemma~\ref{L:Qpms} and the Rellich lemma. In our non-compact setting we need to study the behavior of $Q_{-+}$ at infinity. Using  the cut-off functions of  Section~\ref{SS:cutoff} we write 
\begin{equation}\label{E:decompQ}
	Q_{-+}\ = \ \phi\circ Q_{-+}\circ \psi 
	\ + \ \phi\circ Q_{-+}\circ (1-\psi)
	\ + \ (1-\phi)\circ Q_{-+}\circ \psi
	\ + \ (1-\phi)\circ Q_{-+}\circ (1-\psi).
\end{equation}

We study each summand in the right hand side separately. 

\subsubsection*{The first summand} Since supports of $\phi$ and $\psi$ are compact, the compactness of the first summand follows from Lemma~\ref{L:Qpms} and the  Rellich lemma. 

To study the other summands we first prove (Lemma~\ref{L:commutator}) that the commutator of the spectral projections with a compactly  supported function is compact. We write $A\equiv B$ if the operators $A$ and $B$ are {\em equal modulo compacts}, i.e., if the operator $A-B$ is compact.

\subsubsection*{The second summand}
By Lemma~\ref{L:commutator}, it is enough to show that the second summand in \eqref{E:decompQ} is compact it is enough to show that the operator $\phi\circ Q\circ (1-\psi)$ is compact. But this operator is equal to 0 because of the finite propagation speed property of $Q$. 

\subsubsection*{The third summand}
Let $\tilde{\psi}$ be a compactly supported function, whose restriction to $K'''$ is identically equal to 1.  By finite propagation speed propery we have $Q\circ\psi\ = \ \tilde{\psi}\circ Q\circ \psi$. Using this fact and   Lemma~\ref{L:commutator} it is easy to check (cf. Lemma~\ref{L:thirdterm}) that the third term is equal modulo compacts to $(1-\phi)\tilde{\psi}\circ Q_{-+}\circ \psi$. Hence, it is compact by combination of Lemma~\ref{L:Qpms} and the Rellich lemma. 

\subsubsection*{The forth term} 
The last  term in \eqref{E:decompQ} is supported on $M\backslash([0,1]\times{}K)$. The restriction of $\AA_t$ is independent of $t$. It follows, cf. \eqref{E:extension},  that 
\[
	(1-\phi)\circ Q\circ(1-\psi)\ = \ 
	(1-\phi)\circ e^{i\AA_0}\circ(1-\psi).
\]
Using Lemma~\ref{L:commutator} it is  easy to see that it suffices to prove that the operator 
\[
	P_{(-\infty,0]}^{0} \circ  e^{i\AA_0} \circ  
	  P_{[0,\infty)}^1\circ (1-\psi)
\] 
is compact. This is done by an explicit computation in the proof of Lemma~\ref{L:infinitepart}. 

\medskip
The rest of this section is occupied with the details of the proof of Proposition~\ref{P:compactness}.

\begin{lemma}\label{L:commutator}
If $f$ is a smooth function with compact support on $\Sigma$ then the commutators $[P_{(-\infty,0)}^0,f]$ and $[P_{[0,\infty)}^0,f]$ are compact operators. Similar statements hold for the commutators with $P_{(-\infty,0]}^1$ and $P_{(0,\infty)}^1$.
\end{lemma}

\begin{proof}
We only prove that $[P_{[0,\infty)}^0,f]$ is compact. The proof for the other 3 commutators is analogous. 

Let $\gamma$ be a contour in complex plane going around the non-negative part of the spectrum of $\AA_0$  in counterclockwise direction which is the union of 3 curves: $\gamma_1:= \{r\cdot e^{i\pi/4}:\,\epsilon\le r<\infty\}$, \ $\gamma_2:= \{\epsilon\cdot e^{i\psi}:\,\pi/4\le \psi\le 7\pi/4\}$, and  $\gamma_3:= \{r\cdot e^{7i\pi/4}:\,\epsilon\le r<\infty\}$. 

For $\lambda$ not in the spectrum of $\AA_0$, 
let $R_{\AA_0}(\lambda):= \big(\lambda-\AA_0\big)^{-1}$ denote the resolvent. Since the operator $\AA_0$ is self-adjoint, we have
\begin{equation}\label{E:normR}
	\big\|\,R_{\AA_0}(\lambda)\,\big\| \ \le \  |\IM \lambda\,|^{-1}, 
\end{equation}
Hence, the integral
\[
	\frac{1}{2\pi i}\, \int_{\gamma}\,\lambda^s\, R_{\AA_0}(\lambda)\, d\lambda, 
\]
is absolutely convergent for $\RE s<0$ and, by functional calculus, is equal to $\big(P_{[0,\infty)}^0A_0\big)^s$.

We have
\begin{equation}\label{E:RDf}
	[R_{\AA_0}(\lambda),f] \ = \ R_{\AA_0}(\lambda)\,[\AA_1,f]\,R_{\AA_0}(\lambda)
	\ = \ 
	R_{\AA_0}(\lambda)\,c(df)\, R_{\AA_0}(\lambda).
\end{equation}

Since $df$ has compact support, it follows from Rellich's Lemma that $c(df) R_{\AA_0}(\lambda)$ is compact.  Hence $[R_{\AA_0}(\lambda),f]$ is also compact. It follows from \eqref{E:RDf} and \eqref{E:normR} that 
\[
	\left[\big(P_{[0,\infty)}^{t_0}A_0\big)^s,f\right]\ =\ 
	\frac{1}{2\pi i}\, 
	  \int_{\gamma}\, \lambda^s\,[R_{\AA_0}(\lambda),f]\: d\lambda
\]
is absolutely convergent for $\RE s<1$ and compact. Hence, 
\[
	[\big(P_{[0,\infty)}^0,f]\ = \ 
	\left[\big(P_{[0,\infty)}^0A_0\big)^s,f\right]_{\big|{s=0}}
\]
is compact.
\end{proof}

\begin{lemma}\label{L:secondterm}
The second term in the right hand side of  \eqref{E:decompQ} is compact.
\end{lemma}
\begin{proof}
The finite propagation speed for the solution of the wave equation implies that the support of $Q\big((1-\psi)\bfu_0\big)$ does not intersect support of $\phi$ for all $\bfu_0\in \LS{0}$. Hence, using Lemma~\ref{L:commutator}, we obtain
\[
	\phi\circ Q_{-+}\circ (1-\psi) \ \equiv \ 
	P_{(0,\infty)}^1\circ\big(\, \phi\circ Q\circ(1-\psi)\,\big)\circ 
	P_{(-\infty,0)}\ = \ 0,
\]
where ``$\equiv$" denote equality modulo compact operators. 
\end{proof}

\begin{lemma}\label{L:thirdterm}
The third term in the right hand side of \eqref{E:decompQ} is compact.
\end{lemma}
\begin{proof}
Recall that the compact set $K'''$ was defined in \eqref{E:K'''}.
Let $\tilde{\psi}$ be a compactly supported function on $\Sigma$ whose restriction to $K'''$ is equal to 1. Then the finite propagation speed for the solutions of the wave equation implies that 
$Q\circ\psi = \tilde{\psi}\circ Q\circ \psi$.
Hence,
\begin{multline}\notag
	(1-\phi)\circ Q_{-+} \circ \psi\ \equiv \ 
	P_{(-\infty,0]}^0\circ(1-\phi)\circ Q \circ 
		\psi\circ P_{[0,\infty)}^1
	\\ = \
	P_{(-\infty,0]}^0\circ(1-\phi)\tilde{\psi}\circ Q 
		\circ \psi\circ P_{[0,\infty)}(t_1)
	\ \equiv \ 
	(1-\phi)\tilde{\psi}\circ Q_{-+} \circ \psi.
\end{multline}
The assertion of the lemma follows now from \eqref{E:HcomHloc} and the Rellich lemma.
\end{proof}

\begin{lemma}\label{L:infinitepart}
The last  term in the right hand side of \eqref{E:decompQ} is compact.
\end{lemma}

\begin{proof}
Using Lemma~\ref{L:commutator} we obtain 
\begin{equation}\label{E:1-psi}
	(1-\phi)\circ Q_{-+}\circ (1-\psi)\ \equiv \ 
	P_{(-\infty,0]}^0\circ(1-\phi)\circ Q \circ 
		(1-\psi)\circ P_{[0,\infty)}^1.
\end{equation}
By \eqref{E:extension} and \eqref{E:evolutionQ}, for a section $\bfu_0$ supported outside of $K'$ we have 
\[
	Q\,\bfu_0\ = \ e^{i\AA_0}\,\bfu_0.
\]
Hence, from \eqref{E:1-psi} and  Lemma~\ref{L:commutator} we obtain	
\begin{multline}\label{E:1-psi2}
	(1-\phi)\circ Q_{-+}\circ (1-\psi)\ \equiv \ 
	P_{(-\infty,0]}^{0}\circ(1-\phi) 
	\circ e^{i\AA_0} \circ
	(1-\psi)\circ P_{[0,\infty)}^1
	\\ \equiv \ 
	(1-\phi)\circ 
	 P_{(-\infty,0]}^{0}
	  \circ  e^{i\AA_0} \circ  
	  P_{[0,\infty)}^1\circ (1-\psi).
\end{multline}

Consider the family of operators 
\[
	S(t)\ := \ 
	P_{(-\infty,0]}^{0}
	\circ  e^{i\AA_0} \circ  
	P_{[0,\infty)}^t\circ (1-\psi), \qquad t_0\le t\le t_1.
\]
Then $S(t_0)$ is equal to  $(1-\psi)$ times the projection onto the kernel of $\AA_{t_0}$. Hence, $S(t_0)$ is a compact (even finite rank) operator. We will show that $S(t)$ is compact for all $t$. The family $S(t)$ is not continuous at the points where some eigenvalues of the family $\AA_t$ cross 0. However, since there are finitely many such eigenvalues, $S(t)$ is continuous (and, as we shall see below, even smooth) modulo compacts. To explore this, for each $t_*\in [t_0,t_1]$ fix a contour $\gamma_{t_*}$ as in the proof of Lemma~\ref{L:commutator} which encloses the non-negative spectrum of $\AA_{t_*}$. Then there is $\epsilon>0$ such that for all $t\in (t_*-\epsilon,t_*+\epsilon)$ the spectrum of $\AA_t$ is disjoint from $\gamma_{t_*}$ and there are at most finitely many positive eigenvalues of $\AA_t$ inside $\gamma_{t_*}$. Thus 
\begin{equation}\label{E:dPequiv}
		\frac{1}{2\pi i}\, \int_{\gamma_{t_*}}\,\lambda^s\, R_{\AA_t}(\lambda)\, d\lambda
	\ \equiv \ \big(P_{[0,\infty)}^t\AA_t\big)^s, 
	\qquad s<0, \ \ t_*-\epsilon<t<t_*+\epsilon.
\end{equation}	 
We now compute the derivative of the left hand side of this equation:  
\begin{equation}\label{E:ddtReitz}
	\frac{d}{dt}\,\left(\,
	\frac{1}{2\pi i}\, \int_{\gamma_{t_*}}\,\lambda^s\, R_{\AA_t}(\lambda)\, d\lambda\,\right)
	\ = \ 
	-\frac{1}{2\pi i}\, \int_{\gamma_{t_*}}\,\lambda^s\, R_{\AA_t}(\lambda)\circ
			\frac{d\AA_t}{dt}\circ R_{\AA_t}(\lambda)    \, d\lambda.
\end{equation}
The integral in the right hand side is absolutely convergent for $\RE s<1$. Thus
\begin{equation}\label{E:ddtP}
	\frac{d}{dt}\,P_{[0,\infty)} 
	\ = \ 
		\frac{d}{dt}\,\big(P_{[0,\infty)}^t\AA_t\big)^0
	 \ \equiv \
	-\frac{1}{2\pi i}\, \int_{\gamma_{t_*}}\, 
	R_{\AA_t}(\lambda)\circ
			\frac{d\AA_t}{dt}\circ R_{\AA_t}(\lambda)    \, d\lambda.
\end{equation}
Since $\frac{d\AA_t}{dt}$ is supported outside of the support of $(1-\psi)$, we have $\frac{d\AA_t}{dt}\circ(1-\psi)=0$. Hence, using Lemma~\ref{L:commutator}, we obtain 
\begin{equation}\label{E:ddtrietzpsi}
	\frac{d}{dt}\,P_{[0,\infty)}^t\ \circ (1-\psi)
	\ \equiv \ 
	-\frac{1}{2\pi i}\, \int_{\gamma_{t_*}}\,\lambda^s\, R_{\AA_t}(\lambda)\circ
			\left(\,\frac{d\AA_t}{dt}\circ(1-\psi)\,\right) R_{\AA_t}(\lambda)    \, d\lambda \ = \ 0.
\end{equation}
Hence, $S(t)= P_{(-\infty,0]}^{t_0}\circ e^{i\AA
_0}\circ P_{[0,\infty)}^t\circ (1-\psi) $ is differentiable modulo compacts and its derivative is 0 modulo compacts. Since $S(t_0)$ is compact, it follows that so is $S(t_1)$.
\end{proof}

\section{The APS index formula}\label{S:APSindex}

In this section we show that the Atiyah-Patodi-Singer (APS) boundary value problem for the Lorentzian strongly Callias-type operator $\DD$ is Fredholm and that its index is equal to the index of the operator $Q_{--}$, which we studied in the previous section. We then introduce a Riemanninan Dirac operator $\cDD$, obtained from $\DD$ by the ``Wick rotation". We show that the APS index of $\DD$ is equal to the APS index of $\cDD$. This leads to an explicit formula for the APS index of $\DD$ in terms of the relative eta-invariant introduced in \cite{BrShi17,BrShi17b}.

\subsection{The Atiyah-Patodi-Singer boundary conditions}\label{SS:APSbc}
We define the space 
\begin{equation}\label{E:FEAPS}
	FE^s_{APS}(M,\DD)\ := \ 
	\big\{\,\bfu\in FE^s(M,\DD):\, 
	  P_{[0,\infty)}^{0}\bfu_{0}=0= P_{(-\infty,0]}\bfu_{t_2}\,\big\}
\end{equation}
of {\em finite $s$-energy sections, which satisfy the Atiyah-Patodi-Singer boundary conditions}.

We have the following analogue of Theorem~3.3 of \cite{BarStrohmaier15}:
\begin{theorem}\label{T:APSindex}
The operator
\begin{equation}\label{E:APSindex}
	\DD_{APS}\ := \ \DD|_{FE^0_{APS}(M,\DD)}:\,
	FE^0_{APS}(M,\DD)\ \to \ L^2(M,E)
\end{equation}
is Fredholm and its index satisfies
\begin{equation}\label{E:indAPS=indQ}
	\ind\,\DD^+_{APS}\ = \ \ind\, Q_{--}.
\end{equation}
\end{theorem}
\begin{proof}
For compact $M$ the theorem is proven in \cite[\S3]{BarStrohmaier15}. The proof there only uses the formal properties of operators $\DD$ and $Q$. For non-compact case the same properties are proven in Sections~\ref{S:evolution} and \ref{S:propertiesevolution}. A verbatim repetition of the proof in \cite{BarStrohmaier15} proves the theorem. 
\end{proof}

\begin{remark}\label{R:antiAPS}
As in  \cite{BarStrohmaier15}, we can define the {\em anti-APS space}: 
\begin{equation}\label{E:FEantiAPS}
	FE^s_{aAPS}(M,\DD)\ := \ 
	\big\{\,\bfu\in FE^s(M,\DD):\, 
	  P_{(-\infty,0)}^{0}\bfu_{0}=0= P_{(0,\infty)}\bfu_{t_2}\,\big\}, 
\end{equation}
and the {\em anti-APS boundary problem}:
\begin{equation}\label{E:aAPSindex}
	\DD_{aAPS}\ := \ \DD|_{FE^0_{aAPS}(M,\DD)}:\,
	FE^0_{aAPS}(M,\DD)\ \to \ L^2(M,E).
\end{equation}
It was noted in \cite{BarStrohmaier15} that, if $M$ is compact, quite surprisingly  $\DD_{aAPS}$ is also Fredholm. The same proof shows that this result remains true in our non-compact situation and, as in \cite[Theorem~3.4]{BarStrohmaier15} we obtain
\begin{equation}\label{E:indaAPS=indQ}
	\ind\,\DD_{aAPS}\ = \ \ind\, Q_{++}.
\end{equation}
\end{remark}

\subsection{The Wick rotation}\label{SS:Wick}
We define the ``Wick rotation" of the Lorentzian metric \eqref{E:gM} by $\check{g}:= N^2dt^2+g_t^\Sigma$. This is a complete Riemannian metric on $M$.   Endow $E$ with the Clifford action $\cgamma:TM\to \End(E)$ such that 	
$\cgamma(v)= \gamma(v)$ for $v\in T\Sigma_t$ and $\cgamma(\nu)= i\gamma(\nu)$. 
Then $\cgamma(v)^2= -\check{g}(v,v)$ and $\cgamma(v)$ is skew-adjoint with respect to the Hermitian scalar product $\<\cdot,\cdot\>_E$ on $E$.

The bundle $E=E^+\oplus E^-$  endowed with connection $\n^E$ and Clifford action $\cgamma$ is a Dirac bundle over $(M,\check{g})$. Let $\check{D}:C^\infty(M,\check{S}M\otimes E)\to C^\infty(M,\check{S}M\otimes E)$ be the  Dirac operator associated to connection $\n^E$. This is a self-adjoint elliptic operator on $M$.  As in the Lorentzian case we set  $\cbeta:=\cgamma(\nu)$ and   define 
\begin{equation}\label{E:CalliasEuclid}
	\cDD\ := \ \check{D} \ + \ i\cbeta\otimes \Phi. 
\end{equation}
This is a self-adjoint strongly Callias-type operator in the sense of Definition~2.2 of \cite{BrShi17b}. 

Since the restriction of $\check{g}$ to each hypersurface $\Sigma_t$ is equal to $g_t^\Sigma$, along $\Sigma_t$ we have the following  analogue of \eqref{E:Calliasnear}
\begin{equation}\label{E:WickCalliasnear}
	\cDD\ =  \ 
	-\cbeta\,\Big(\,\n^{\check{S}M\otimes E}_\nu  + i\,\AA_t
	 - \frac{n}{2}\check{H}_t\,\Big).
\end{equation}
In particular, the restriction $\AA_t$ of $\cDD$ to $\Sigma_t$ coincides with the restriction of $\DD$. 

As in \eqref{E:restrictionDD}, we see that the restriction of $\cDD$ to $M\backslash\big([0,1]\times K\big)$ is equal to 
\begin{equation}\label{E:restrictionWickDD}
	\cDD|_{M\backslash\big([0,1]\times K\big)}\ = \ 
	-\beta\, \Big(-\frac{\p}{\p t}  + i\,\AA_0\,\Big).
\end{equation}

\subsection{The APS index of the elliptic Callias-type operator}\label{SS:APScDD}
Let $\cDD^+$ denote the restriction of $\cDD$ to $E^+$ and let  $\dom\cDD^+_{\max}$ denote the domain of the maximal extension of $\cDD^+$ (cf. \cite[\S2.2]{BrShi17} for more details).  
We denote by $\cDD_{APS}$ the restriction of $\cDD$ to the space of sections $\bfu\in \dom\cDD^+_{\max}$, satisfying the APS boundary conditions: $P_{[0,\infty)}^{0}\bfu_{0}=0= P_{(-\infty,0]}\bfu_{t_2}$
\begin{equation}\label{E:cDDAPS}
	\cDD^+_{APS}:\, 
	\big\{\, \bfu\in \dom\cDD^+_{\max}:\, 
	  P_{[0,\infty)}^{0}\bfu_{0}=0= P_{(-\infty,0]}\bfu_{t_2}\,\big\}
	\ \to \ 
	L^2(M,S^-M\otimes E).
\end{equation}
If the manifold $M$ is compact, then it is well-known that the opertor $\cDD^+_{APS}$ is Fredholm, cf., for example, \cite{BarBallmann12}. The case of non-compact manifolds was studied in \cite{BrShi17,BrShi17b}. Using an extension of the method of \cite{BarBallmann12} it is shown in \cite{BrShi17} that $\cDD^+_{APS}$ is Fredholm if $\cDD$ is a product near the boundary of $M$. 
Combining the arguments in \cite{BarBallmann12} and \cite{BrShi17} one immediately sees that $\cDD^+_{APS}$ is Fredholm if $\cDD$ is a product outside of a compact set in a neighborhood of the boundary of $M$.  Hence, by condition (ii) of Subsection~\ref{SS:glhyperbolic}, our operator $\cDD^+_{APS}$ is Fredholm. 
 
The main result of this section is the following

\begin{theorem}\label{T:indcheckD=indD}
$\displaystyle\ind\DD^+_{APS}\ = \ \ind\cDD^+_{APS}$.
\end{theorem}
The proof of the theorem occupies Sections~\ref{SS:APSformulacDDproduct}-\ref{SS:prindcheckD=indD}.

\subsection{The APS index fromula for an elliptic operator: product caes}\label{SS:APSformulacDDproduct}
Assume, first, that $\cDD$ is a product near the boundary, i.e., 
\begin{equation}\label{E:productcaseWickDD}
	\cDD|_{[0,\epsilon)\times\Sigma}\ = \ 
	-\beta\, \Big(-\frac{\p}{\p t}  + i\,\AA_0\,\Big),
	\qquad
	\cDD|_{(1-\epsilon,1]\times\Sigma}\ = \ 
	-\beta\, \Big(-\frac{\p}{\p t}  + i\,\AA_1\,\Big).
\end{equation}
Then $\cDD$ is an {\em almost compact cobordism between $\AA_0$ and $\AA_1$} in the sense of Definition~4.2 of \cite{BrShi17b}. In particular, $\AA_0$ and $\AA_1$ coincide outside of the compact set $[0,1]\times K$. Hence, by \cite[Eq.~(4.1)]{BrShi17b} we obtain
\begin{equation}\label{E:indcDDproduct}
	\ind \cDD^+_{APS}\ = \ 
	\int_M\,\alpha_{AS}(\check{\DD}^+) \ +\ 
	\frac{\eta(\AA_0,\AA_1)-\dim\ker\AA_0-\dim\ker\AA_1}{2},
\end{equation}
where 
\[
	\alpha_{AS}(\cDD)\ := \ 
	(2\pi i)^{-\dim M}\,\hat{A}(M,\check{g})\cdot\ch(E/S)
\]
is the Atiyah-Singer integrand of $\cDD$ and $\eta(\AA_0,\AA_1)$ is the invariant of the pair $(\AA_0,\AA_1)$ introduced in Definition~4.4 of \cite{BrShi17b} and called the {\em relative eta-invariant of $(\AA_0,\AA_1)$}. Since $\cDD$ is product outside of the compact set $[0,1]\times K$, the Atiyah-Singer integrand vanishes outside of this set and,  hence, the integral in the right hand side of \eqref{E:indcDDproduct} is well defined. 

Morally, the relative eta invariant $\eta(\AA_0,\AA_1)$ is the difference of the eta-invariants of $\AA_1$ and $\AA_0$ but the later invariants might not be defined in non-compact case. However, it is shown in \cite{BrShi17b} that, in many respects,  $\eta(\AA_0,\AA_1)$  behaves like it were the difference. In particular, 
\[
	\eta(\AA_0,\AA_1)\ = \ -\,\eta(\AA_1,\AA_0), 
	\qquad
	\eta(\AA_0,\AA_1) \ + \ \eta(\AA_2,\AA_1)\ = \ \eta(\AA_0,\AA_2).
\]

Further, suppose $\AAA:= \{\AA_1^t\}_{0\le t\le 1}$ is a smooth family of strongly Callias-type operator, whose restriction to $M\backslash\big([0,1]\times K\big)$ is independent of $t$. Then the spectral flow $\spf(\AAA)$ is well defined, cf. \cite[Definition~5.7]{BrShi17b}. Then, \cite[Theorem~5.10]{BrShi17b}, the mod $\ZZ$ reduction $\oeta(\AA_1^t,\AA_0)$ of $\eta(\AA_1^t,\AA_0)$ depends smoothly on $t$ and 
\begin{equation}\label{E:spflowandeta}
	2\spf(\AAA)\ = \ 
	\eta(\AA_1^1,\AA_0)\ - \ \eta(\AA_1^0,\AA_0)\ - \ 
	\int_0^1\,\big(\frac{d}{ds}\oeta(\AA_1^s,\AA_0)\big)\,ds. 
\end{equation}

\subsection{The APS index fromula for an elliptic operator: general caes}\label{SS:APSformulacDDgeneral}
Consider now the general case when $\cDD$ is not necessary a product near the boundary of $M$ (recall, however, that we always assume that $\cDD$ is a product outside of a compact set). The method of computing the index of the APS boundary problem in this case is due to Gilkey, \cite{Gilkey75,Gilkey93}. The idea is to deform all the data (the metric $\check{g}$, the connection $\n^E$, the potential $\Phi$) to those which are product near the boundary. Thus we obtain a smooth family of first order elliptic operators $\cDD^s$ ($0\le s\le 1$) such that $\cDD^1= \cDD$ and $\cDD^0$ is product near the boundary. Of course, we assume that the restriction of $\cDD^s$ to $M\backslash([0,1]\times K)$ is independent of $s$. 

By Chern-Weil  theory there is a {\em transgression differential form} $\TAS(\cDD^+,\cDD^{0+})$, such that 
\begin{equation}\label{E:transgression3}
	\alpha_{AS}(\cDD^+)\ - \ \alpha_{AS}(\cDD^{0+})
	\ = \ d\TAS(\cDD^+,\cDD^{0+}).
\end{equation}
This differential form is given by an explicit formula in terms of $\check{g}^s$ and $\n^{E,s}$, cf. \cite[Proposition~1.41]{BeGeVe} or \cite[\S6]{BrMaschler17}. In particular,  this form vanishs outside of $[0,1]\times K$. 

By Stokes formula, 
\begin{multline}\label{E:inttransgression}
	\int_M\,\alpha_{AS}(\cDD)\ = \ \int_M\alpha_{AS}(\cDD^0)
	\ + \ \int_M  d\TAS(\cDD^+,\cDD^{0+}) 
	\\ = \ 
	\int_M\alpha_{AS}(\cDD^0)
	\ + \ \int_{\Sigma_1}  \TAS(\cDD^+,\cDD^{0+}) 
	\ - \ \int_{\Sigma_0}  \TAS(\cDD^+,\cDD^{0+}).
\end{multline}
Recall that the forms $\alpha_{AS}(\cDD^+)$, $\alpha_{AS}(\cDD^{0+})$, $\TAS(\cDD^+,\cDD^{0+})$ vanish outside of a compact subset of $M$. Hence all the integrals in \eqref{E:inttransgression}	are well defined. 

Combining \eqref{E:inttransgression} with \eqref{E:indcDDproduct} and using the stability of the index, $\ind\cDD^+= \ind\cDD^{0+}$, we obtain
\begin{multline}\label{E:generalindex}
	\ind\cDD^+_{APS}\ = \ 
		\int_M\,\alpha_{AS}(\check{\DD}^+) 
		\ + \ \int_{\Sigma_1}  \TAS(\cDD^+,\cDD^{0+}) 
	\ - \ \int_{\Sigma_0}  \TAS(\cDD^+,\cDD^{0+})
	\\ + \ 	
	\frac{\eta(\AA_0,\AA_1)-\dim\ker\AA_0-\dim\ker\AA_1}{2},
\end{multline}

\subsection{Index and the spectral flow}\label{SS:indspflow}
Consider the family of operators $\AAA:=\{\AA_t\}_{0\le t\le 1}$ and let $\spf(\AAA)$ denotes its spectral flow.

\begin{proposition}\label{P:ind=spf}
The following equalities hold
\begin{equation}\label{E:indDD=spf}
	\ind\cDD^+_{APS}\ = \ \spf(\AAA) \ - \ \dim\ker(\AA_1)
	\ = \ 
	\ind Q_{--}.
\end{equation}
\end{proposition}
\begin{proof}
For the case when $M$ is compact the proposition is proven in Sections~4.1-4.2 of \cite{BarStrohmaier15}. The proof there only uses the properties of the  right hand side of \eqref{E:generalindex}. A verbatim repetition of this proof (using \eqref{E:spflowandeta} instead of corresponding equations for $\eta(A_1)$ and $\eta(A_0)$ in \cite{BarStrohmaier15}) proves the proposition. 
\end{proof}

\subsection{Proof of Theorem~\ref{T:indcheckD=indD}}\label{SS:prindcheckD=indD}
Combining Theorem~\ref{T:APSindex} with Proposition~\ref{P:ind=spf} we obtain Theorem~\ref{T:indcheckD=indD}. \hfill$\square$

As an immediate corollary of Theorem~\ref{T:indcheckD=indD} we obtain
\begin{corollary}\label{C:geometricindexformilar}
The APS index $\ind\DD^+_{APS}$ of the Lorentzian strongly Callias-type operator $\DD^+$ is given by the right hand side of \eqref{E:generalindex}. 
\end{corollary}


\end{document}